\documentclass[12]{amsart}
\usepackage{amsmath,amssymb,amsthm,color,enumerate,comment,centernot,enumitem,url,cite}
\usepackage{graphicx,relsize,bm}
\usepackage{mathtools}
\usepackage{array}

\makeatletter
\newcommand{\tpmod}[1]{{\@displayfalse\pmod{#1}}}
\makeatother

\newtheorem{thm}{Theorem}[section]
\newtheorem{lemma}[thm]{Lemma}
\newtheorem{conj}[thm]{Conjecture}
\newtheorem{prop}[thm]{Proposition}
\newtheorem{cor}[thm]{Corollary}

\theoremstyle{remark}

\theoremstyle{definition}
    \newtheorem{defn}[thm]{Definition}

\newtheorem{rem}[thm]{Remark}

\theoremstyle{THM}

\newcommand{\abs}[1]{\left|{#1}\right|}

\def\FF {{\mathcal F}}

\def\Z {{\mathbb Z}}

\def\NN {{\mathcal N}}

\def\Q {{\mathbb Q}}

\def\C {{\mathcal C}}

\def\D {{\mathcal D}}

\def\F {{\mathbb F}}
\def\D {{\mathcal D}}
\def\Z {{\mathbb Z}}
\def\Q {{\mathbb Q}}
\def\C {{\mathbb C}}

\makeatletter
\@namedef{subjclassname@2020}{%
  \textup{2020} Mathematics Subject Classification}
\makeatother

\def\red#1 {\textcolor{red}{#1 }}
\def\blue#1 {\textcolor{blue}{#1 }}

\numberwithin{equation}{section}

\def\Z {{\mathbb Z}}
\begin{document}

\title[Power Compositional Trinomials]{The Irreducibility and Monogenicity of Power Compositional Trinomials}

\author{Joshua Harrington}
\address{Department of Mathematics, Cedar Crest College, Allentown, Pennsylvania, USA}
\email[Joshua Harrington]{Joshua.Harrington@cedarcrest.edu}

\author{Lenny Jones}
\address{Professor Emeritus, Department of Mathematics, Shippensburg University, Shippensburg, Pennsylvania 17257, USA}
\email[Lenny~Jones]{lkjone@ship.edu}

\date{\today}

\begin{abstract}
A polynomial $f(x)\in \Z[x]$ of degree $N$ is called \emph{monogenic} if $f(x)$ is irreducible over $\Q$ and $\{1,\theta,\theta^2,\ldots ,\theta^{N-1}\}$ is a basis for the ring of integers of $\Q(\theta)$, where $f(\theta)=0$. Define $\FF(x):=x^m+Ax^{m-1}+B$.
In this article, we determine sets of conditions on $m$, $A$, and $B$, such that the power compositional trinomial $\FF(x^{p^n})$ is monogenic for all integers $n\ge 0$ and a given prime $p$. Furthermore, we prove the actual existence of infinite families of such trinomials $\FF(x)$.
\end{abstract}


\maketitle
\section{Introduction}\label{Section:Intro}

 Unless stated otherwise, when we say that $f(x)$ is ``irreducible", we mean irreducible over $\Q$. We let $\Delta(f(x))$, or simply $\Delta(f)$, and $\Delta(K)$ denote the discriminants over $\Q$, respectively, of $f(x)$ and a number field $K$.
If $f(x)$ is irreducible, with $f(\theta)=0$ and $K=\Q(\theta)$, then \cite{Cohen}
\begin{equation} \label{Eq:Dis-Dis}
\Delta(f)=\left[\Z_K:\Z[\theta]\right]^2\Delta(K),
\end{equation}
where $\Z_K$ is the ring of integers of $K$. We define $f(x)$ to be \emph{monogenic} if $f(x)$ is irreducible and
$\Z_K=\Z[\theta]$, or equivalently from \eqref{Eq:Dis-Dis}, that  $\Delta(f)=\Delta(K)$. When $f(x)$ is monogenic, $\{1,\theta,\theta^2,\ldots ,\theta^{\deg(f)-1}\}$ is a basis for $\Z_K$, commonly referred to as a \emph{power basis}. The existence of a power basis facilitates computations in $\Z_K$, as in the case of the cyclotomic polynomials $\Phi_n(x)$ \cite{Washington}.
We see from \eqref{Eq:Dis-Dis} that if $\Delta(f)$ is squarefree, then $f(x)$ is monogenic. However, the converse is false in general, and when $\Delta(f)$ is not squarefree, it can be quite difficult to determine whether $f(x)$ is monogenic.

In this article, our focus is on trinomials of the form
\begin{equation}\label{F}
\begin{array}{c}
  \FF(x):=\FF_{m,A,B}(x)=x^{m}+Ax^{m-1}+B\in \Z[x].
  \end{array}
\end{equation} We determine sets of conditions on $m$, $A$ and $B$ so that the power compositional trinomial $\FF(x^{p^n})$ is monogenic (and hence irreducible)  for all integers $n\ge 0$ and a given prime $p$. We then prove the actual existence of infinite families of such trinomials.
      Although much research has been conducted concerning the irreducibility and monogenicity of trinomials \cite{BMT,DS,Josh,JonesIJM,JonesACM,JonesJAA,JonesNYJM,JP, KR,PR,SSS,Shparlinski,Smith}, the approaches presented in this article are novel, and they allow us to construct new infinite families of monogenic trinomials in a manner unlike any methods previously used. Our main results are as follows.

 \begin{thm}\label{Thm:Main1}
 Let $\FF(x)$ be as defined in \eqref{F}, 
 and let $p\ge 3$ be a prime. Define
  \begin{equation}\label{Eq:D}
  \D:=m^mB-(-1)^m(m-1)^{m-1}A^m.
  \end{equation}
  \begin{enumerate}
     \item \label{I1:m=2} When $m=2$, let $A=4pu+p^2+2$ and $B=2pt+1$, 
         where $u,t\in \Z$ are such that $B$, $B-1$ and $\D$ are squarefree, with $\abs{B}\ge 2$.
    \item \label{I2:m=3 and m=4} When $m\in \{3,4\}$, let $A=4p^2u+1$ and $B=2pt+p$, where $u,t\in \Z$ are such that $B$ and $\D$ are squarefree, with $\abs{B}\ge 2$.
    \item \label{I3:m>=5} When $m\ge 5$, let $A=4p^2u+1$ and $B=2pt+p$, where $u,t\in \Z$ are such that $B$ and $\D$ are squarefree, with $\abs{B}\ge 2$, and assume that $\FF(x)$ is irreducible.
  \end{enumerate}
  Then, in any case above, $\FF(x^{p^n})$ is monogenic for all $n\ge 0$.
  \end{thm}
\begin{cor}\label{Cor:Main1}
 Let $\FF(x)$ be as defined in \eqref{F}, 
  let $p\ge 3$ be a prime, and let $u\in \Z$. Let $\D$ be as defined in \eqref{Eq:D}. 
 Then, in any case of Theorem \ref{Thm:Main1}, there exist infinitely many prime values of $t$ such that $\FF(x^{p^n})$ is monogenic for all $n\ge 0$.
\end{cor}

\section{Preliminaries}\label{Section:Prelim}

The formula for the discriminant of an arbitrary monic trinomial, due to Swan \cite{Swan}, is given in the following theorem.
\begin{thm}
\label{Thm:Swan}
Let $f(x)=x^n+Ax^m+B\in \Q[x]$, where $0<m<n$, and let $d=\gcd(n,m)$. Then
\[
\Delta(f)=(-1)^{n(n-1)/2}B^{m-1}\left(n^{n/d}B^{(n-m)/d}-(-1)^{n/d}(n-m)^{(n-m)/d}m^{m/d}A^{n/d}\right)^d.
\]
\end{thm}
The following immediate corollary of Theorem \ref{Thm:Swan} is of pertinence to us here.
\begin{cor}\label{Cor:Swan}
  Let $\FF(x)$ be as defined in \eqref{F}, and let $p$ be a prime. Then, for any integer $n\ge 0$, we have
  \begin{equation}\label{Eq:Disc(F(x^{p^n})}
  \Delta(\FF(x^{p^n}))=(-1)^{mp^n(mp^n-1)/2}B^{(m-1)p^n-1}p^{mnp^n}\D^{p^n},
  \end{equation}
  where $\D$ is as defined in \eqref{Eq:D}.
 \end{cor}

The next two theorems are due to Capelli \cite{S}.
 \begin{thm}\label{Thm:Capelli1}  Let $f(x)$ and $h(x)$ be polynomials in $\Q[x]$ with $f(x)$ irreducible. Suppose that $f(\alpha)=0$. Then $f(h(x))$ is reducible over $\Q$ if and only if $h(x)-\alpha$ is reducible over $\Q(\alpha)$.
 \end{thm}

\begin{thm}\label{Thm:Capelli2}  Let $c\in \Z$ with $c\geq 2$, and let $\alpha\in\C$ be algebraic.  Then $x^c-\alpha$ is reducible over $\Q(\alpha)$ if and only if either there is a prime $p$ dividing $c$ such that $\alpha=\beta^p$ for some $\beta\in\Q(\alpha)$ or $4\mid c$ and $\alpha=-4\beta^4$ for some $\beta\in\Q(\alpha)$.
\end{thm}
The following theorem, known as \emph{Dedekind's Index Criterion}, or simply \emph{Dedekind's Criterion} if the context is clear, is a standard tool used in determining the monogenicity of a polynomial.
\begin{thm}[Dedekind \cite{Cohen}]\label{Thm:Dedekind}
Let $K=\Q(\theta)$ be a number field, $T(x)\in \Z[x]$ the monic minimal polynomial of $\theta$, and $\Z_K$ the ring of integers of $K$. Let $q$ be a prime number and let $\overline{ * }$ denote reduction of $*$ modulo $q$ (in $\Z$, $\Z[x]$ or $\Z[\theta]$). Let
\[\overline{T}(x)=\prod_{i=1}^k\overline{\tau_i}(x)^{e_i}\]
be the factorization of $T(x)$ modulo $q$ in $\F_q[x]$, and set
\[g(x)=\prod_{i=1}^k\tau_i(x),\]
where the $\tau_i(x)\in \Z[x]$ are arbitrary monic lifts of the $\overline{\tau_i}(x)$. Let $h(x)\in \Z[x]$ be a monic lift of $\overline{T}(x)/\overline{g}(x)$ and set
\[F(x)=\dfrac{g(x)h(x)-T(x)}{q}\in \Z[x].\]
Then
\[\left[\Z_K:\Z[\theta]\right]\not \equiv 0 \pmod{q} \Longleftrightarrow \gcd\left(\overline{F},\overline{g},\overline{h}\right)=1 \mbox{ in } \F_q[x].\]
\end{thm}

The next result is essentially a ``streamlined" version of Dedekind's index criterion for trinomials.
\begin{thm}{\rm \cite{JKS2}}\label{Thm:JKS}
Let $N\ge 2$ be an integer.
Let $K=\Q(\theta)$ be an algebraic number field with $\theta\in \Z_K$, the ring of integers of $K$, having minimal polynomial $f(x)=x^{N}+Ax^M+B$ over $\Q$, with $\gcd(M,N)=d_0$, $M=M_1d_0$ and $N=N_1d_0$. A prime factor $q$ of $\Delta(f)$ does not divide $\left[\Z_K:\Z[\theta]\right]$ if and only if all of the following statements are true: 
\begin{enumerate}[font=\normalfont]
  \item \label{JKS:I1} if $q\mid A$ and $q\mid B$, then $q^2\nmid B$;
  \item \label{JKS:I2} if $q\mid A$ and $q\nmid B$, then
  \[\mbox{either } \quad q\mid A_2 \mbox{ and } q\nmid B_1 \quad \mbox{ or } \quad q\nmid A_2\left((-B)^{M_1}A_2^{N_1}-\left(-B_1\right)^{N_1}\right),\]
  where $A_2=A/q$ and $B_1=\frac{B+(-B)^{q^j}}{q}$ with $q^j\mid\mid N$;
  \item \label{JKS:I3} if $q\nmid A$ and $q\mid B$, then
  \[\mbox{either } \quad q\mid A_1 \mbox{ and } q\nmid B_2 \quad \mbox{ or } \quad q\nmid A_1B_2^{M-1}\left((-A)^{M_1}A_1^{N_1-M_1}-\left(-B_2\right)^{N_1-M_1
  }\right),\]
  where $A_1=\frac{A+(-A)^{q^\ell}}{q}$ with $q^\ell\mid\mid N-M$, and $B_2=B/q$;
  \item \label{JKS:I4} if $q\nmid AB$ and $q\mid M$ with $N=s^{\prime}q^k$, $M=sq^k$, $q\nmid \gcd\left(s^{\prime},s\right)$, then the polynomials
   \begin{equation*}
     x^{s^{\prime}}+Ax^s+B \quad \mbox{and}\quad \dfrac{Ax^{sq^k}+B+\left(-Ax^s-B\right)^{q^k}}{q}
   \end{equation*}
   are coprime modulo $q$;
         \item \label{JKS:I5} if $q\nmid ABM$, then
     \[q^2\nmid \left(B^{N_1-M_1}N_1^{N_1}-(-1)^{M_1}A^{N_1}M_1^{M_1}(M_1-N_1)^{N_1-M_1} \right).\]
   \end{enumerate}
\end{thm}
\begin{rem}
We will find both Theorem \ref{Thm:Dedekind} and Theorem \ref{Thm:JKS} useful in our investigations.
\end{rem}

The next theorem follows from Corollary (2.10) in \cite{Neukirch}.
\begin{thm}\label{Thm:CD}
  Let $K$ and $L$ be number fields with $K\subset L$. Then \[\Delta(K)^{[L:K]} \bigm|\Delta(L).\]
\end{thm}

\begin{thm}\label{Thm:Pasten}
   Let $G(t)\in \Z[t]$, and suppose that $G(t)$ factors into a product of distinct irreducibles, such that the degree of each irreducible is at most 3.   Define 
   \[N_G\left(X\right)=\abs{\left\{p\le X : p \mbox{ is prime and } G(p) \mbox{ is squarefree}\right\}}.\]
   Then, 
   \[N_G(X)\sim C_G\dfrac{X}{\log(X)},\]
   where
   \[C_G=\prod_{\ell  \mbox{ \rm {\tiny prime}}}\left(1-\dfrac{\rho_G\left(\ell^2\right)}{\ell(\ell-1)}\right)\]
   and $\rho_G\left(\ell^2\right)$ is the number of $z\in \left(\Z/\ell^2\Z\right)^{*}$ such that $G(z)\equiv 0 \pmod{\ell^2}$.
 \end{thm}

\begin{rem}
  Theorem \ref{Thm:Pasten} follows from work of Helfgott, Hooley and Pasten \cite{Helfgott-cubic,Hooley-book,Pasten}. For more details, see \cite{JonesNYJM}.
\end{rem}

\begin{defn}\label{Def:Obstruction}
 In the context of Theorem \ref{Thm:Pasten}, for $G(t)\in \Z[t]$ and a prime $\ell$, if $G(z)\equiv 0 \pmod{\ell^2}$ for all $z\in \left(\Z/\ell^2\Z\right)^{*}$, we say that $G(t)$ has a \emph{local obstruction} at $\ell$. 
\end{defn}
The following immediate corollary of Theorem \ref{Thm:Pasten} is used to establish Corollary \ref{Cor:Main1}. 
\begin{cor}\label{Cor:Squarefree}
 Let $G(t)\in \Z[t]$, and suppose that $G(t)$ factors into a product of distinct irreducibles, such that the degree of each irreducible is at most 3. To avoid the situation when $C_G=0$, we suppose further that $G(t)$ has no local obstructions.
  Then there exist infinitely many primes $\rho$ such that $G(\rho)$ is squarefree.
\end{cor}

We make the following observation concerning $G(t)$ from Corollary \ref{Cor:Squarefree} in the special case when each of the distinct irreducible factors of $G(t)$ is of the form $a_it+b_i$ with $\gcd(a_i,b_i)=1$. In this situation, it follows that the minimum number of distinct factors required in $G(t)$ so that $G(t)$ has a local obstruction at the prime $\ell$ is $2(\ell-1)$. More precisely, in this minimum scenario, we have
\[G(t)=\prod_{i=1}^{2(\ell-1)}(a_it+b_i)\equiv C(t-1)^2(t-2)^2\cdots (t-(\ell-1))^2 \pmod{\ell},\] where $C\not \equiv 0 \pmod{\ell}$. Then  each zero $r$ of $G(t)$ modulo $\ell$ lifts to the $\ell$ distinct zeros
\[r,\quad r+\ell,\quad  r+2\ell,\quad \ldots \ldots, \quad r+(\ell-1)\ell \in \left(\Z/\ell^2\Z\right)^{*}\] of $G(t)$ modulo $\ell^2$ \cite[Theorem 4.11]{Dence}. That is, $G(t)$ has exactly $\ell(\ell-1)=\phi(\ell^2)$ distinct zeros $z\in \left(\Z/\ell^2\Z\right)^{*}$. Therefore, if the number of factors $k$ of $G(t)$ satisfies $k<2(\ell-1)$, then there must
exist $z\in \left(\Z/\ell^2\Z\right)^{*}$ for which $G(z)\not \equiv 0 \pmod{\ell^2}$, and we do not need to check such primes $\ell$ for a local obstruction. Consequently, only finitely many primes need to be checked for local obstructions. They are precisely the primes $\ell$ such that $\ell\le (k+2)/2$.

\section{The Proof of Theorem \ref{Thm:Main1}}\label{Section:MainProof}
Before we begin the proof of Theorem \ref{Thm:Main1}, we require the following. 

\begin{lemma}\label{Lem:1}
  Let $f(x)\in \Z[x]$ be monic and irreducible, and let $p$ be a prime.
  If $\abs{f(0)}\ge 2$ and $f(0)$ is squarefree, then $f(x^{p^n})$ is irreducible for all $n\ge 0$.
\end{lemma}
\begin{proof}
   Suppose that $f(\alpha)=0$. Let $n\ge 1$ and assume, by way of contradiction, that $f(g(x))$ is reducible, where $g(x)=x^{p^n}$. We deduce from Theorems \ref{Thm:Capelli1} and \ref{Thm:Capelli2} that $\alpha=\beta^p$ for some $\beta \in \Q(\alpha)$. Then we conclude, by taking the norm $\NN:=\NN_{\Q(\alpha)/\Q}$, that
     \[\NN(\beta)^p=\NN(\alpha)=(-1)^{\deg{(f)}}f(0),\]
    which contradicts the fact that $f(0)$ is squarefree since $\NN(\beta)\in \Z$ and $\abs{f(0)}\ge 2$.
 \end{proof}

Note that the conditions in part \eqref{I3:m>=5} of Theorem \ref{Thm:Main1} contain the assumption that $\FF(x)$ is irreducible. Under this assumption, we obtain the following immediate corollary of Lemma \ref{Lem:1}.
 \begin{cor}\label{Cor:m>=5}
Let $m\in \Z$ with $m\ge 3$ and let $p$ be a prime with $p\ge 3$. Let $\FF(x)$ be as defined in \eqref{F}, and let $\D$ be as defined in \eqref{Eq:D}. Suppose that $A=4p^2u+1$ and $B=2pt+p$, where $u,t\in \Z$ are such that $B$ and $\D$ are
squarefree, with $\abs{B}\ge 2$. If $\FF(x)$ is irreducible, then $\FF(x^{p^n})$ is irreducible for all $n\ge 0$.
\end{cor}
 Computer evidence suggests the following, which, if true, implies the conclusion of Corollary \ref{Cor:m>=5} by Lemma \ref{Lem:1}.
\begin{conj}\label{Conj:Irreducible}
Let $m\ge 3$, and let $\FF(x)$ be as defined in \eqref{F}. If $A\equiv 1 \pmod{4}$ and $B\equiv 1 \pmod{2}$, with $B$ squarefree and $\abs{B}\ge 2$, then $\FF(x)$ is irreducible.
\end{conj}

 The next lemma is similar in nature to Lemma 3.1 in \cite{JonesBAMS}. We let $T$ be the set of the 12 possible ordered pairs $(a,b)$, where $0\le a \le 3$ and $1\le b\le 3$, corresponding respectively to the possible congruence classes modulo 4 of $A$ and $B$ (excluding $b=0$ since $B$ is squarefree). For any integer $z$, we let $\widehat{z}\in \{0,1,2,3\}$ be the reduction of $z$ modulo 4.
 \begin{lemma}\label{Lem:2}
   Let $\FF(x)$ be as defined in \eqref{F} with $m\in \{2,3,4\}$, and let $p$ be a prime. Suppose that $B$ is squarefree with $\abs{B}\ge 2$. Then $\FF(x^{p^n})$ is irreducible for all $n\ge 0$ if  $(\widehat{A},\widehat{B})\in \Gamma_m$, where
     \[\Gamma_m=\left\{\begin{array}{cl}
      \{(0,1),(0,2),(1,1),(1,3), (2,2),(2,3),(3,1),(3,3)\} & \mbox{if $m\in \{2,4\}$}\\[.5em]
      \{(0,2),(1,1),(1,3),(2,2),(3,1),(3,3)\}& \mbox{if $m=3$.}
      \end{array}\right.\]
            Moreover, for each $m\in \{2,3,4\}$, and each $(a,b)\in T\setminus \Gamma_m$, there exist integers $A$ and $B$ with $(\widehat{A},\widehat{B})=(a,b)$ such that $\FF(x^{p^n})$ is reducible for all $n\ge 0$.
 \end{lemma}
 \begin{proof}
   We begin by showing, for each $m\in \{2,3,4\}$ and each $(\widehat{A},\widehat{B})\in \Gamma_m$, that $\FF(x)$ is irreducible. It is easy to check that $\FF(x)$ has no zeros modulo 4, which verifies that $\FF(x)$ is irreducible when $m\in\{2,3\}$. When $m=4$, we also need to show that $\FF(x)$ cannot be written as the product of two irreducible quadratics. Suppose then that
   \[\FF(x)=x^4+Ax^3+B=(x^2+c_1x+c_0)(x^2+d_1x+d_0),\] for some $c_i,d_i\in \Z$. Expanding this factorization and equating coefficients yields the system of equations:
    \begin{align}\label{Eq:n=0 m=4}
  \begin{split}
  \mbox{constant term}: & \quad c_0d_0=B\\
   x:& \quad c_0d_1+c_1d_0=0\\
   x^2:& \quad c_0+d_0+c_1d_1=0\\
   x^3: & \quad c_1+d_1=A.
   \end{split}
 \end{align}
It turns out that for each $(\widehat{A},\widehat{B})\in \Gamma_4$, the system \eqref{Eq:n=0 m=4} is impossible modulo 4. For example, if $(\widehat{A},\widehat{B})=(3,1)$, then
\[\mbox{either } c_0\equiv d_0\equiv 1 \tpmod{4} \quad \mbox{or} \quad c_0\equiv d_0\equiv 3 \tpmod{4}.\] In either of these cases, we see in \eqref{Eq:n=0 m=4} that the congruence corresponding to the coefficient on $x$ contradicts the congruence corresponding to the coefficient on $x^3$. Similar contradictions are easily seen to occur for each $(\widehat{A},\widehat{B})\in \Gamma_4$, and we omit the details. It follows that $\FF(x^{p^n})$ is irreducible for all $n\ge 0$ by Lemma \ref{Lem:1}.

   To complete the proof, for each $m\in \{2,3,4\}$ we give in Tables \ref{Table:1}, \ref{Table:2} and \ref{Table:3}, respectively, an example of each  $(\widehat{A},\widehat{B})\in T\setminus \Gamma_m$ such that $\FF(x^{p^n})$ is reducible for all $n\ge 0$. We also give the factorization of $\FF(x^{p^n})$.\vspace*{.25in}

   \begin{table}[h]
 \begin{center}
\begin{tabular}{ccc}
 $(\widehat{A},\widehat{B})\in T\setminus \Gamma_2$ & $(A,B)$ & Factorization of $\FF(x^{p^n})$\\ \hline \\[-8pt]
$(0,3)$ & $(4,3)$ & $(x^{p^n}+1)(x^{p^n}+3)$\\ [2pt]
$(1,2)$ & $(5,6)$ & $(x^{p^n}+2)(x^{p^n}+3)$\\ [2pt]
$(2,1)$ & $(2,1)$  & $(x^{p^n}+1)^2$\\ [2pt]
$(3,2)$ & $(3,2)$ & $(x^{p^n}+1)(x^{p^n}+2)$
   \end{tabular}
\end{center}
\caption{Examples for $(\widehat{A},\widehat{B})\in T\setminus \Gamma_2$ and their factorizations}
 \label{Table:1} 
\end{table}\vspace*{.25in}

 \begin{table}[h]
 \begin{center}
\begin{tabular}{ccc}
 $(\widehat{A},\widehat{B})\in T\setminus \Gamma_3$ & $(A,B)$ & Factorization of $\FF(x^{p^n})$\\ \hline \\[-8pt]
$(0,1)$ & $(-4,5)$ & $(x^{p^n}+1)(x^{2p^n}-5x^{p^n}+5)$\\ [2pt]
$(0,3)$ & $(-4,3)$ & $(x^{p^n}-1)(x^{2p^n}-3x^{p^n}-3)$\\ [2pt]
$(1,2)$ & $(-3,2)$  & $(x^{p^n}-1)(x^{2p^n}-2x^{p^n}-2)$\\ [2pt]
$(2,1)$ & $(-2,1)$ & $(x^{p^n}-1)(x^{2p^n}-x^{p^n}-1)$\\ [2pt]
$(2,3)$ & $(-2,3)$ & $(x^{p^n}+1)(x^{2p^n}-3x^{p^n}+3)$\\ [2pt]
$(3,2)$ & $(-1,2)$ & $(x^{p^n}+1)(x^{2p^n}-2x^{p^n}+2)$
   \end{tabular}
\end{center}
\caption{Examples for $(\widehat{A},\widehat{B})\in T\setminus \Gamma_3$ and their factorizations}
 \label{Table:2} 
\end{table}\vspace*{.25in}

 \begin{table}[h]
 \begin{center}
\begin{tabular}{ccc}
 $(\widehat{A},\widehat{B})\in T\setminus \Gamma_4$ & $(A,B)$ & Factorization of $\FF(x^{p^n})$\\ \hline \\[-8pt]
$(0,3)$ & $(4,3)$ & $(x^{p^n}+1)(x^{3p^n}+3x^{2p^n}-3x^{p^n}+3)$\\ [2pt]
$(1,2)$ & $(-3,2)$ & $(x^{p^n}-1)(x^{3p^n}-2x^{2p^n}-2x^{p^n}-2)$\\ [2pt]
$(2,1)$ & $(2,1)$  & $(x^{p^n}+1)(x^{3p^n}+x^{2p^n}-x^{p^n}+1)$\\ [2pt]
$(3,2)$ & $(3,2)$ & $(x^{p^n}+1)(x^{3p^n}+2x^{2p^n}-2x^{p^n}+2)$
   \end{tabular}
\end{center}
\caption{Examples for $(\widehat{A},\widehat{B})\in T\setminus \Gamma_4$ and their factorizations}
 \label{Table:3}\qedhere
\end{table}
\end{proof}

Computer evidence suggests the following extension of Lemma \ref{Lem:2}.
\begin{conj}\label{Conj:Gamma}
Let $\FF(x)$ be as defined in \eqref{F}, and let $p$ be a prime. Suppose that $B$ is squarefree with $\abs{B}\ge 2$. Then $\FF(x^{p^n})$ is irreducible for all $n\ge 0$ if
  $(\widehat{A},\widehat{B})\in \Gamma_m$, where
    \[\Gamma_m=\left\{\begin{array}{cl}
      \{(0,1),(0,2),(1,1),(1,3),(2,2),(2,3),(3,1),(3,3)\} & \mbox{if $m\equiv 0\pmod{2}$}\\[.5em]
      \{(0,2),(1,1),(1,3),(2,2),(3,1),(3,3)\}& \mbox{if $m\equiv 1 \pmod{2}.$}
      \end{array}\right.\]
\end{conj}

\begin{rem}
  Note that the truth of Conjecture \ref{Conj:Gamma} also implies the conclusion of Corollary \ref{Cor:m>=5}.
\end{rem}

\begin{lemma}\label{Lem:3}
  Let $\FF(x)$ be as defined in \eqref{F}, and let $p$ be a prime. Suppose that $B$ is squarefree with $\abs{B}\ge 2$, and that $\D$, as defined in \eqref{Eq:D}, is squarefree. If $\FF(x)$ is irreducible, then $\FF(x)$ is monogenic.
  \end{lemma}
\begin{proof}
We have by Corollary \ref{Cor:Swan} that
      \begin{align*}
      \Delta(\FF(x))&=(-1)^{m(m-1)/2}B^{m-2}\left(m^mB-(-1)^m(m-1)^{m-1}A^m\right)\\
      &=(-1)^{m(m-1)/2}B^{m-2}\D.
      \end{align*} Let $K=\Q(\theta)$, where $\FF(\theta)=\theta^m+A\theta^{m-1}+B=0$,
 and let $\Z_K$ denote the ring of integers of $K$. We check that all statements in Theorem \ref{Thm:JKS} are true with $N:=m$ and $M:=m-1$, and all primes $q$ dividing $\Delta(\FF(x))$. Note that statements \eqref{JKS:I1} and \eqref{JKS:I5} are always true for any value of $q$ since, respectively, $B$ and $\D$ are squarefree.

  Suppose first that $q\mid B$.  We see immediately that statements \eqref{JKS:I2} and \eqref{JKS:I4} are trivially true. For statement \eqref{JKS:I3}, we have $\ell=0$ so that $A_1=0$. Hence, $q\mid A_1$, and $q\nmid B_2$ since $B$ is squarefree. Thus, statement \eqref{JKS:I3} is true and  $\left[\Z_K:\Z[\theta]\right]\not \equiv 0 \pmod{q}.$

  Suppose now that $q$ is a prime such that $q\mid \D$. Observe that if $q\mid A$ and $q\nmid B$, then $q\mid m$, which contradicts the fact that $\D$ is squarefree. Hence, statement \eqref{JKS:I2} is trivially true. If $q\nmid A$ and $q\mid B$, then $\ell=0$ so that $A_1=0$. Hence, $q\mid A_1$, and $q\nmid B_2$ since $B$ is squarefree. Therefore, statement \eqref{JKS:I3} is true. Suppose next that $q\nmid AB$ and $q\mid (m-1)$. Then $q\mid m$ since $q\mid \D$, which is impossible, and so statement \eqref{JKS:I4} is trivially true.  Thus, $\left[\Z_K:\Z[\theta]\right]\not \equiv 0 \pmod{q}$, and consequently, $\FF(x)$ is monogenic, which completes the proof of the lemma.
\end{proof}

\begin{proof}[Proof of Theorem \ref{Thm:Main1}]
For part \eqref{I1:m=2}, we see that 
\[(\widehat{A},\widehat{B})\in \{(3,1),(3,3)\}\subset \Gamma_2.\]
Thus, $\FF(x^{p^n})$ is irreducible for all $n\ge 0$ by Lemma \ref{Lem:2}, and $\FF(x)$ is monogenic by Lemma \ref{Lem:3}.
By Corollary \ref{Cor:Swan}, we have that
\begin{equation}\label{Eq:Discp^n m=2}
\Delta(\FF(x^{p^n}))=(-1)^{p^n(2p^n-1)}B^{p^n-1}p^{2np^n}\D^{p^n},
\end{equation}
where $\D=4B-A^2$.
 Define
\[\theta_n:=\theta^{1/p^n} \quad \mbox{and} \quad K_n:=\Q(\theta_n) \quad \mbox{for $n\ge 0$},\] noting that $\theta_0=\theta$ and $K_0=K$ from the proof of Lemma \ref{Lem:3}. Furthermore, observe that $\FF((\theta_n)^{p^n})=0$ and
$[K_{n+1}:K_n]=p$ for all $n\ge 0$. We assume that $\FF(x^{p^n})$ is monogenic and proceed by induction on $n$. Then $\Delta\left(\FF(x^{p^n})\right)=\Delta(K_n)$, and we deduce from Theorem \ref{Thm:CD} that
\[\Delta\left(\FF(x^{p^n})\right)^p\bigm| \Delta(K_{n+1}).\]
By \eqref{Eq:Discp^n m=2}, we have that
\[\Delta(\FF(x^{p^{n+1}}))/\Delta(\FF(x^{p^n}))^p=B^{p-1}p^{2p^{n+1}}.\] Hence, to show that $\FF(x^{p^{n+1}})$ is monogenic, we only have to show that
\begin{equation}\label{Eq:modq}
\left[\Z_{K_{n+1}}:\Z[\theta_{n+1}]\right]\not \equiv 0 \pmod{q}
\end{equation}
for all primes $q$ dividing $Bp$. We check that all statements in Theorem \ref{Thm:JKS} are true for such primes, when applied to $\FF(x^{p^{n+1}})=x^{2p^{n+1}}+Ax^{p^{n+1}}+B$.

Suppose first that $q=p$. We use Theorem \ref{Thm:Dedekind} with $T(x):=\FF(x^{p^{n+1}})$. Since $A\equiv 2\pmod{p}$ and $B\equiv 1 \pmod{p}$, we get that
$\overline{T}(x)=(x+1)^{2p^{n+1}}$. Thus, we can let $g(x)=x+1$ and $h(x)=(x+1)^{2p^{n+1}-1}$, so that
 \begin{align}\label{Eq:F(x)}
   F(x)&=\dfrac{g(x)h(x)-T(x)}{p}\nonumber \\
   &=\dfrac{(x+1)^{2p^{n+1}}-x^{2p^{n+1}}-Ax^{p^{n+1}}-B}{p}\nonumber \\
   &=\sum_{\stackrel{j=1}{j\ne p^{n+1}}}^{2p^{n+1}-1}\dfrac{\binom{2p^{n+1}}{j}}{p}x^j+\dfrac{\binom{2p^{n+1}}{p^{n+1}}-A}{p}x^{p^{n+1}}+\dfrac{1-B}{p}.
 \end{align} To show that $\gcd(\overline{g},\overline{F})=1$, it is enough to show that $F(-1) \not \equiv 0 \pmod{p}$. Since
 \[\sum_{\stackrel{j=1}{j\ne p^{n+1}}}^{2p^{n+1}-1}\binom{2p^{n+1}}{j}x^j=(x+1)^{2p^{n+1}}-x^{2p^{n+1}}-1-\binom{2p^{n+1}}{p^{n+1}}x^{p^{n+1}},\]
 it follows that
 \begin{equation*}\label{Eq:Binom}
 \sum_{\stackrel{j=1}{j\ne p^{n+1}}}^{2p^{n+1}-1}\binom{2p^{n+1}}{j}(-1)^j=-2+\binom{2p^{n+1}}{p^{n+1}}\equiv 0\pmod{p^2} \quad \mbox{ \cite{Bailey,Grinberg}.}
 \end{equation*}
Hence, the sum in \eqref{Eq:F(x)} is zero modulo $p$ when $x=-1$. Also, since $A\equiv 2\pmod{p}$ and $B-1$ is squarefree, we see from \eqref{Eq:F(x)} that
 \[F(-1)\equiv \overline{\dfrac{1-B}{p}}\not \equiv 0 \pmod{p}.\] Thus, we conclude from Theorem \ref{Thm:Dedekind} that \eqref{Eq:modq} is true for $q=p$.

  Suppose next that $q\mid B$. Note that $q\ne p$. Here we use Theorem \ref{Thm:JKS}, and we want to show that all statements are true for $\FF(x^{p^{n+1}})$ with $N=2p^{n+1}$ and $M=p^{n+1}$. Statement \eqref{JKS:I1} is true since $B$ is squarefree. We see that statements \eqref{JKS:I2}, \eqref{JKS:I4} and \eqref{JKS:I5} are trivially true. For statement \eqref{JKS:I3}, we have that $\ell=0$, since $q\ne p$ and $N-M=p^{n+1}$, and hence, $A_1=0$. Thus, $q\mid A_1$ and $q\mid B_2$ since $B$ is squarefree. Therefore, \eqref{Eq:modq} is true. Consequently, $\FF(x^{p^n})$ is monogenic for all $n\ge 0$ by induction.

  For parts \eqref{I2:m=3 and m=4} and \eqref{I3:m>=5}, we give details only for the case $m=3$ since the other cases are handled in an identical manner. In particular, for part \eqref{I3:m>=5} (cases of $m\ge 5$), we use Corollary \ref{Cor:m>=5} and proceed as in the case of $m=3$.

  For the case of $m=3$, we see that 
\[(\widehat{A},\widehat{B})\in \{(1,1),(1,3)\}\subset \Gamma_3.\]
Thus, $\FF(x^{p^n})$ is irreducible for all $n\ge 0$ by Lemma \ref{Lem:2}, and $\FF(x)$ is monogenic by Lemma \ref{Lem:3}.
By Corollary \ref{Cor:Swan}, we have that
\begin{equation}\label{Eq:Discp^n m=3}
\Delta(\FF(x^{p^n}))=(-1)^{3p^n(3p^n-1)/2}B^{2p^n-1}p^{3np^n}\D^{p^n},
\end{equation}
where $\D=27B+4A^3$. Using \eqref{Eq:Discp^n m=3} and Theorem \ref{Thm:CD} for
\[\FF(x^{p^{n+1}})=x^{3p^{n+1}}+Ax^{2p^{n+1}}+B,\] and arguing as in part \eqref{I1:m=2} with $\theta_n$ and $K_n$ defined accordingly, we deduce that we only have to check primes dividing $B$, since $B\equiv 0 \pmod{p}$ in this case.

Suppose first that $q=p$. We use Theorem \ref{Thm:Dedekind} with $T(x):=\FF(x^{p^{n+1}})$. Since $A\equiv 1\pmod{p}$ and $B\equiv 0 \pmod{p}$, we get that
\[\overline{T}(x)=x^{2p^{n+1}}(x+1)^{p^{n+1}}.\] Thus, we may let $g(x)=x(x+1)$ and $h(x)=x^{2p^{n+1}-1}(x+1)^{p^{n+1}-1}$, so that
 \begin{align}\label{Eq:F(x)m=3}
   F(x)&=\dfrac{g(x)h(x)-T(x)}{p}\nonumber \\
   &=\dfrac{x^{2p^{n+1}}(x+1)^{p^{n+1}}-x^{3p^{n+1}}-Ax^{2p^{n+1}}-B}{p}\nonumber \\
   &=\sum_{j=1}^{p^{n+1}-1}\dfrac{\binom{p^{n+1}}{j}}{p}x^{j+2p^{n+1}}+\left(\dfrac{1-A}{p}\right)x^{2p^{n+1}}-\dfrac{B}{p}.
 \end{align} To show that $\gcd(\overline{g},\overline{F})=1$, it is enough to show that $F(c) \not \equiv 0 \pmod{p}$, where $c\in \{-1,0\}$. Note first that since $A\equiv 1 \pmod{p^2}$, the middle term in \eqref{Eq:F(x)m=3} disappears in $\overline{F}$. Observe next that the sum in \eqref{Eq:F(x)m=3} is zero when $x=-1$. Hence,
 \[F(c)\equiv-\overline{\dfrac{B}{p}}\not \equiv 0 \pmod{p},\] since $B$ is squarefree.  Thus, we deduce from Theorem \ref{Thm:Dedekind} that \eqref{Eq:modq} is true for $q=p$.

  Suppose next that $q\mid B$ with $q\ne p$. Here we use Theorem \ref{Thm:JKS}, and we want to show that all statements are true for $\FF(x^{p^{n+1}})$ with $N=3p^{n+1}$ and $M=2p^{n+1}$. Statement \eqref{JKS:I1} is true since $B$ is squarefree. We see that statements \eqref{JKS:I2}, \eqref{JKS:I4} and \eqref{JKS:I5} are trivially true. For statement \eqref{JKS:I3}, we have that $\ell=0$, since $q\ne p$ and $N-M=p^{n+1}$, and hence, $A_1=0$. Thus, $q\mid A_1$ and $q\mid B_2$ since $B$ is squarefree. Therefore, \eqref{Eq:modq} is true. Consequently, $\FF(x^{p^n})$ is monogenic for all $n\ge 0$ by induction.
 \end{proof}

\section{Proof of Corollary \ref{Cor:Main1}}
\begin{proof}
For all cases of $m$ in Theorem \ref{Thm:Main1}, we define $G(t):=B\D$. We wish to apply Corollary \ref{Cor:Squarefree}, and so we must check for local obstructions. In light of the discussion following Corollary \ref{Cor:Squarefree}, since $B$ and $\D$ are both linear in $t$, we only have to check for a local obstruction at the prime $\ell=2$.

Specifically, for part \eqref{I1:m=2} of Theorem \ref{Thm:Main1} (when $m=2$), we get that
  \[G(t)=p(2pt+1)(8t-16pu^2-8p^2u-16u-4p-p^3).\]
   Since
  $G(1)\equiv 1\pmod{4}$, there is no local obstruction at $\ell=2$. Thus, by Corollary \ref{Cor:Squarefree}, there exist infinitely many primes $\rho$ such that $G(\rho)$ is squarefree. Therefore, we see that $\abs{B}\ge 2$, and both $B$ and $\D$, individually, are squarefree for such primes $\rho$. Note also that $B-1=2pt$ is squarefree when $t=\rho>p$. Consequently, for each $t=\rho>p$, $\FF(x^{p^n})$ is monogenic for all integers $n\ge 0$ by Theorem \ref{Thm:Main1}.

When $m\ge 3$, we get generically that
\[G(t)=p(2t+1)(m^m(2pt+p)-(-1)^m(m-1)^{m-1}(4p^2u+1)^m),\] from which it follows easily for $m\ge 3$ that
\[G(1)\equiv \left\{\begin{array}{cl}
 3p \pmod{4}& \mbox{if $m\equiv 0 \pmod{4}$}\\
 1 \pmod{4} & \mbox{if $m\equiv 1 \pmod{4}$}\\
 p \pmod{4}& \mbox{if $m\equiv 2 \pmod{4}$}\\
 3 \pmod{4}& \mbox{if $m\equiv 3 \pmod{4}$.}
 \end{array}\right.\]
 Therefore, in each of these cases, we conclude from Corollary \ref{Cor:Squarefree} that there exist infinitely many primes $\rho$ such that $G(\rho)$ is squarefree. Consequently, for each $t=\rho$, $\FF(x^{p^n})$ is monogenic for all integers $n\ge 0$ by Theorem \ref{Thm:Main1}.
\end{proof}

\end{document}